\documentclass[12pt,a4paper]{amsart}
\usepackage{mathptmx}
\usepackage{eucal}
\usepackage{graphicx}
\usepackage{mathrsfs}
\usepackage{amssymb}
\usepackage{amsmath}
\usepackage{amsthm}
\usepackage{amscd}
\usepackage{hyperref}
\newtheorem{thm}{Theorem}[section]
\newtheorem{cor}[thm]{Corollary}
\newtheorem{lem}[thm]{Lemma}

\theoremstyle{definition}
\newtheorem{defn}[thm]{Definition}
\newtheorem{rem}[thm]{Remark}
\numberwithin{equation}{section}

\newcommand{\N}{\mathbb{N}}

\begin{document}
\title{Equivalent conditions of Devaney chaos on the hyperspace}
\author{Jian Li}
\address{Department of Mathematics, Shantou University, Shantou, Guangdong 515063, P.R. China}
\email{lijian09@mail.ustc.edu.cn}

\begin{abstract}
  Let $T$ be a continuous self-map of a compact metric space $X$. The transformation $T$ induces natural
  a continuous self-map $T_K$ on the hyperspace $K(X)$ of all non-empty closed subsets of $X$.
  In this paper, we show that the system $(K(X),T_K)$ on the hyperspace is Devaney chaotic if and only if
  $(K(X),T_K)$ is an HY-system if and only if $(X,T)$ is an HY-system,
  where a system $(Y,S)$ is called an HY-system if it is  totally transitive and has dense small periodic sets.
\end{abstract}

\maketitle

\section{introduction}
Throughout this paper, by a \emph{topological dynamical system} $(X,T)$ we mean a compact metric space $X$
with a continuous map $T$ from $X$ into itself; the metric on $X$ is denoted by $d$.

Let $K(X)$ be the hyperspace on $X$, i.e., the space of non-empty closed subset of $X$ equipped with
the Hausdorff metric $d_H$ defined by
\[d_H(A,B)=\max\Bigl\{\max_{x\in A}\min_{y\in B} d(x,y),\ \max_{y\in B}\min_{x\in A}
d(x,y)\Bigr\}\ \text{for}\ A,B\in K(X).\]
The transformation $T$ induces natural a continuous self-map $T_K$ on the hyperspace $K(X)$ defined by
\[T_K(C)=TC,\ \text{for}\ C\in K(X).\]
Then $(K(X),T_K)$ is also a topological dynamical system.

In 1975, Bauer and Sigmund~\cite{BS75} initiated the study on connection between
properties of  $(X,T)$ and $(K(X),T_K)$.
They considered such properties as distality, transitivity and mixing property.
This leads to a natural question: if one of the system $(X,T)$ and $(K(X),T_K)$ is chaotic in some sense,
how about the other? This question attracted a lot of attention, see, e.g., \cite{B05,GKLOP09,Roman03}
and references therein, and many partial answers were obtained.

One of the most popular definition of chaos is Devaney chaos introduced in~\cite{D89} (see Section~2).
The aim of this paper is to obtain some equivalent conditions of Devaney chaos on the hyperspace.
The main result is the following:
\begin{thm}[Main result]\label{thm:main-result}
Let $(X,T)$ be a dynamical system with $X$ being infinite. Then the following conditions are equivalent:
\begin{enumerate}
  \item $(K(X),T_K)$ is Devaney chaotic;
  \item $(K(X),T_K)$ is an HY-system;
  \item $(X,T)$ is an HY-system.
\end{enumerate}
\end{thm}

\section{Preliminaries}
\subsection{Basic definitions and notations}
Let $\N$ denote the set of positive integers.
Through the rest of this paper $X$ denotes an infinite compact metric space with a metric $d$.

Let $(X,T)$ be a dynamical system.
We say that $x\in X$ is a \emph{periodic point} if $T^nx=x$ for some $n\in\N$.
The set of all periodic points of the system $(X,T)$ is denoted as $Per(X,T)$.

\begin{defn}
Let $(X,T)$ be a dynamical system. The system $(X,T)$ is called
\begin{enumerate}
  \item \emph{transitive} if any two non-empty open subsets $U$ and $V$ of $X$,
  there exits an $n\in\N$ such that $T^n U\cap V\neq\emptyset$;
  \item \emph{totally transitive} if for each $n\in\N$, $(X,T^n)$ is transitive;
  \item \emph{weakly mixing} if the product system $(X\times X,T\times T)$ is transitive;
\end{enumerate}
\end{defn}

\begin{defn}
Let $(X,T)$ be a dynamical system. The system $(X,T)$ is called \emph{Devaney chaotic} if
it satisfies the following conditions:
\begin{enumerate}
  \item it is transitive;
  \item periodic points are dense;
  \item it is sensitive to initial conditions.
\end{enumerate}
\end{defn}
It is well known that if $X$ is infinite then the third condition of Devaney chaos is
redundant (see, e.g.,~\cite{BBCDS92}). Since we will restrict our attention to compact
metric spaces without isolated points we will say that $(X,T)$ is Devaney chaotic
if it is transitive with dense periodic points.

\begin{defn}
Let $(X,T)$ be a dynamical system. We say that $(X,T)$ has \emph{dense small periodic sets}
if for any non-empty open subset $U$ of $X$
there exists a closed subset $Y$ of $U$ and $k \in \N$ such that $T^kY\subset Y$.
\end{defn}

\begin{defn}
Let $(X,T)$ be a dynamical system. The system $(X,T)$ is called
\begin{enumerate}
  \item an \emph{F-system} if it is totally transitive and has a dense set periodic points.
  \item an \emph{HY-system} if it is totally transitive and  has dense small periodic sets.
\end{enumerate}
\end{defn}

In~\cite{F67} Furstenberg showed that every $F$-system is weakly mixing and disjoint from any minimal system.
Recently, Huang and Ye~\cite{HY05} showed that a system which is totally transitive and
has dense small periodic sets is also weakly mixing and disjoint from any minimal system.
For this reason, such a system is called an \emph{HY-system} in~\cite{Li11}.
It should be noticed that the author in~\cite{Li11} characterized HY-system
by transitive points via the family of weakly thick sets.
Clearly, every F-system is also an HY-system, but there exists an HY-system without periodic points~\cite{HY05}.

\subsection{Hyperspaces}
Let $K(X)$ be the hyperspace of $X$, i.e., the space of non-empty closed subset of $X$ equipped with
the Hausdorff metric $d_H$. Then $K(X)$ is also a compact mertic space.
The following family
\[
\{\langle U_1, \cdots, U_n\rangle:\; U_1, \cdots, U_n \textrm{ are non-empty open subsets of } X,\ n\in \mathbb{N}\}
\]
forms a basis for a topology of $K(X)$ called the \emph{Vietoris topology}, where
\[
\langle S_1,\cdots, S_n\rangle\doteq \Bigl\{A\in K(X):\; K\subset
\bigcup_{i=1}^n S_i \textrm{ and } A\cap S_i\neq \emptyset \textrm{ for each }i=1,\cdots, n\Bigr\}
\]
is defined for arbitrary non-empty subsets $S_1,\cdots, S_n\subset X$.
It is not hard to see that the Hausdorff topology (the topology induced by the Hausdorff metric $d_H$)
and the Vietoris topology for $K(X)$ coincide. For more details on hyperspaces see~\cite{N78}.

\section{Proof of the main result}
First, we need the following useful Lemma.
\begin{lem}[\cite{BS75,B05}]\label{lem:WM-tran}
Let $(X,T)$ be a dynamical system. Then the following conditions are equivalent:
\begin{enumerate}
  \item $(K(X),T_K)$ is weakly mixing;
  \item $(K(X),T_K)$ is transitive;
  \item $(X,T)$ is weakly mixing.
\end{enumerate}
\end{lem}
Now we turn to the proof of the main result of this paper.
\begin{proof}[Proof of Theorem~\ref{thm:main-result}]
(1)$\Rightarrow$(2)  follows from the definitions and~Lemma~\ref{lem:WM-tran}.

(2)$\Rightarrow$(3) Assume that $(K(X),T_K)$ is an HY-system. By Lemma~\ref{lem:WM-tran},
$(X,T)$ is weakly mixing. Then we only need to show that $(X,T)$ has dense small periodic sets.
Let $U$ be a non-empty open subset of $X$.
Then $\langle U \rangle=\{A\in K(X):\; A\subset U\}$ is an open subset of $K(X)$.
Since $(K(X),T_K)$ has dense small periodic set, there exists a closed subset
$\mathcal{A}\subset \langle U\rangle$ and $k\in\N$ such that $(T_K)^k\mathcal{A}\subset \mathcal{A}$.
Let $Y=\bigcup\{A:\; A\in\mathcal{A}\}$.
We want to show that $Y$ is a closed subset of $X$.
Let $x_n$ be a sequence of points in $Y$ and converge to $x$.
Then for every $n\in\mathbb{N}$ there exists an $A_n\in\mathcal{A}$ such that $x_n\in A_n$.
Since $\mathcal{A}$ is a closed subset of $K(X)$ and $K(X)$ is compact,
without loss of generality, we can assume that the sequence $A_n$ converge to $A\in \mathcal{A}$ in the hyperspace $K(X)$.
By the definition of Hausdorff metric, it is not hard to check that $x\in A$.
Then $x\in Y$, which implies that $Y$ is a closed subset of $X$.
Clearly, $Y\subset U$ and $T^k Y\subset Y$. Thus, $(X,T)$ has dense small periodic sets.

(3)$\Rightarrow$(1)  Assume that $(X,T)$ is an HY-system. Then $(X,T)$ is weakly mixing, and
by Lemma~\ref{lem:WM-tran} $(K(X),T_K)$ is also weakly mixing. So we are left to show that
$(K(X),T_K)$ has dense set of periodic points. Let $\mathcal{U}$ be a non-empty open subset of $K(X)$.
Then there exist pairwise disjointness open subsets $U_1,\dotsc,U_n$ of $X$ such that
$\langle U_1,\dotsc,U_n\rangle \subset \mathcal{U}$.
Since $(X,T)$ has dense small periodic sets, there exist $Y_1,\dotsc,Y_n$ and $k_1,\dots,k_n$
such that $Y_i\subset U_i$ and $T^{k_i}Y_i\subset Y_i$ for $i=1,\dotsc,n$.
For each $i=1,\dotsc,n$, choose a closed subset $Z_i$ of $Y_i$ with $T^{k_i}Z_i=Z_i$.
Let $Z=\bigcup_{i=1}^n Z_i$ and $k=k_1\times \dotsm\times k_n$. Then $Z\in K(X)$ and $(T_K)^k Z=Z$.
Therefore, $(K(X),T_K)$ is chaotic in the Devaney sense.
\end{proof}

In~\cite{KM05}, Kwietniak and Misiurewicz strengthened the notion of Devaney chaos in a sense to the extreme,
and introduce exact Devaney chaos.
Recall that a system $(X,T)$ is called \emph{topologically exact}
if for any non-empty open subset $U$ of $X$ there exists an $n\in\N$ such that $T^n U=X$.
The system $(X,T)$  is called \emph{exactly Devaney chaotic}
if it is topologically exact with dense periodic points.

\begin{lem}[\cite{GKLOP09}]\label{lem:exact}
Let $(X,T)$ be a dynamical system. Then
$(K(X),T_K)$ is topologically exact if and only if $(X,T)$ is topologically exact.
\end{lem}

\begin{cor}
Let $(X,T)$ be a dynamical system. Then $(K(X),T_K)$ is exactly Devaney chaotic if and only if
$(X,T)$ is a topologically exact HY-system.
\end{cor}
\begin{proof}
It  follows from Theorem~\ref{thm:main-result} and Lemma~\ref{lem:exact}.
\end{proof}

\begin{rem}
(1) The authors of~\cite{GKLOP09} constructed a system $(X,T)$ such that $(K(X),T_K)$ is exactly Devaney chaotic, while
the set of  periodic points $Per(X,T)$ is nowhere dense. But this system does have periodic points.

(2) The authors of~\cite{HY05} constructed an HY-system $(X,T)$ without periodic points.
Then $(K(X),T_K)$ is Devaney chaotic, while the set of  periodic points $Per(X,T)$ is empty.
\end{rem}
\subsection*{Acknowledgement}
This work was supported in part by Shantou University Scientific Research Foundation for Talents.
The author wish to thank Wen Huang and Dominik Kwietniak
for the careful reading and helpful suggestions.

\end{document}